\documentclass[reqno]{amsart}       % onecolumn (second format)
\usepackage{graphicx}
\usepackage{amsmath,amssymb}    
\usepackage{mathrsfs}    
\usepackage{verbatim}
\usepackage[backref,colorlinks]{hyperref}  

\usepackage{subcaption}
\usepackage[english]{babel}
\usepackage[utf8]{inputenc}
\usepackage{amsmath,amssymb}
\usepackage{algorithm}
\usepackage{algpseudocode}
\usepackage{geometry}
\usepackage{xspace}

\usepackage{amsmath,amsfonts,amssymb,amsopn,amsthm}
\usepackage{enumerate}
\theoremstyle{plain}
\usepackage{array}

\newcommand{\sk}{s_k(G)}
\newcommand{\ak}{a_k(G)}
\newcommand{\CCk}{\mathcal{C}_k}
\newcommand{\FF}{\mathcal F}

\newtheorem{theorem}{Theorem}
\newtheorem{thm}[theorem] {Theorem}
\newtheorem{lemma}[theorem]{Lemma}

\newtheorem{corollary}[theorem]{Corollary}

\newtheorem{conj}[theorem]{Conjecture}

\title[Coloring of Graphs Avoiding  Bicolored Paths of a Fixed Length]
{Coloring of Graphs Avoiding  Bicolored Paths}

\author[Alaittin K\i rt\i \c{s}o\u{g}lu]{Alaittin K\i rt\i \c{s}o\u{g}lu}
\address{Hacettepe University, Department of Mathematics, Beytepe 06810 Ankara, Turkey. \\
              \emph{Illinois Institute of Technology, Department of Applied Mathematics,  Chicago, IL 60616 USA.} }
\email{alaittinkirtisoglu@gmail.com}

\author[L. \"Ozkahya]{Lale \"Ozkahya}
\address{Hacettepe University, Department of Computer Engineering, Beytepe 06810 Ankara, Turkey.}
\email{lale.ozkahya@gmail.com}

\date{\today}

\begin{document}

\maketitle

\begin{abstract}
The problem of finding the minimum number 
of colors to color a graph properly without containing 
any bicolored copy of a fixed family of subgraphs has been 
widely studied. 
Most well-known examples are star coloring 
and acyclic coloring of graphs (Gr\"unbaum, 1973) 
where bicolored copies of $P_4$ and cycles are not allowed, 
respectively. 
In this paper, we introduce a variation of these problems and 
study proper coloring of graphs not containing a bicolored path 
of a fixed length and provide general bounds for all graphs.  
A $P_k$-coloring of an undirected graph $G$ is a proper vertex
coloring of $G$ such that there is no bicolored copy of $P_k$ in $G,$ 
and the minimum number of colors needed for a $P_k$-coloring 
of $G$ is called the {\it $P_k$-chromatic number} of $G,$ 
denoted by $s_k(G).$ 
We provide bounds on $s_k(G)$ for all graphs, in particular, 
proving that for any graph $G$ with maximum degree $d\geq 2,$ and 
$k\geq4,$ $s_k(G)\leq \lceil 6\sqrt{10}d^{\frac{k-1}{k-2}} \rceil.$
Moreover, we find the exact values for the $P_k$-chromatic number 
of the products of some cycles and paths for $k=5,6.$  
%\subclass{05C35 \and 05C15 \and 05C38}
\end{abstract}

\section{Introduction} \label{sec:introduction}

The proper coloring problem on graphs seeks to find colorings on 
vertices with minimum number of colors 
such that no two neighbors receive the same color. 
There have been studies introducing additional conditions to 
proper coloring, such as forbidding 2-colored copies 
of some graphs. In particular, 
{\it star coloring} problem on a graph $G$ 
asks to find the minimum number of colors in a proper 
coloring forbidding a 2-colored $P_4,$ called the star-chromatic 
number $\chi_s(G)$~\cite{grunbaum1973acyclic}. 
Similarly, {\it acyclic chromatic number} of a graph $G$, $a(G)$, 
is the minimum number of colors used in a proper coloring 
not having any 2-colored cycle, also called acyclic coloring of $G$
~\cite{grunbaum1973acyclic}. Both, the star coloring and 
acyclic coloring problems are shown to be NP-complete 
Albertson et al.~\cite{albertson2004coloring} and Kostochka~\cite{kostochka1978upper}, respectively. 
These two problems have been studied widely on many different families 
on graphs such as product of graphs, particularly grids and hypercubes. 

Acyclic coloring was also introduced in 1973 by Gr\"{u}nbaum \cite{grunbaum1973acyclic} who proved that a graph with maximum degree 3 has an acyclic coloring with 4 colors. 
Alon et al.~\cite{alon1991acyclic} prove that there exist graphs $G$ 
with maximum $d$ for which $ a(G)=\Omega\left(\frac{d^{\frac{4}{3}}}{{(logd)}^\frac{1}{3}}\right).$ 
In ~\cite{alon1991acyclic}, it is also shown that for any graph $G$ 
with maximum degree $d,$ $a(G)= O(d^\frac{4}{3}).$ 
%Alon, McDiarmid, and Reed \cite{alon1991acyclic} proved that acyclic chromatic number of any graph $G$ with maximum degree $d$ is equal or smaller than $\lceil 50d^{4/3} \rceil$ using the probabilistic method Lovasz Local Lemma. 
Recently, there have been some improvements in the constant factor of the 
upper bound in~\cite{esperet2013acyclic,gonccalves2014entropy,ndreca2012improved}  
by using the entropy compression method. 
Similar results for the star chromatic number of graphs 
are obtained by Fertin et al.~\cite{fertin2004star}, 
showing $\chi_s(G)\leq \lceil 20d^{3/2} \rceil$ 
for any graph $G$ with maximum degree $d.$ 

Moreover, in \cite{aravind2011bounds,aravind2013forbidden} and \cite{gonccalves2014entropy}, more general bounds are 
shown introducing the chromatic number of {\it (2,$\FF$)-subgraph coloring}, defined as 
a proper coloring of the vertex set such that there is no bicolored copy 
of any subgraph $H\in \FF.$ The minimum number of colors needed 
to obtain a  $(2,\FF)$-subgraph coloring of a graph $G$ is denoted 
by $\chi_{2,\FF}(G).$ In particular, $\chi_{2,\FF}(\Delta)$ is used to denote 
the maximum value of $\chi_{2,\FF}(G)$ for any graph $G$ with maximum 
degree $\Delta.$ 
Aravind and Subramanian \cite{aravind2011bounds} show a lower bound on $\chi_{2,\FF}(\Delta)$ for the case $\FF = \{H\},$ where $H$ is a connected bipartite graph with $m$ edges ($m \geq 2$) as 
\[
\chi_{2,\{H\}}(\Delta) = \Omega 
\left( 
\frac{\Delta^{\frac{m}{m-1}}}{(\log \Delta )^{\frac{1}{m-1}}}
\right). 
\] 
This yields the same lower bound for $\chi_{2,\FF}(\Delta)$ for any family $\FF$ of connected bipartite graphs, where
$m$ is the minimum number of edges in any member of $\FF.$ For any such family $\FF,$ 
Aravind and Subramanian \cite{aravind2013forbidden}
prove that 
%for any family of connected bipartite graphs $\FF$ with all graphs having  at least $m$ edges, 
\[
\chi_{2,\FF}(\Delta) \leq C\Delta^{\frac{m}{m-1}},
\]
where $C$ is some constant depending only on $\FF$. 
Gon\c{c}alves et al. \cite{gonccalves2014entropy} improve this result introducing the parameters 
$k_v^{\leq m}$ and $k_e^m$, denoting the number of members in $\FF$ 
with at most $m$ vertices and with exactly $m$ edges, respectively.   
They show
\begin{equation}\label{gonc1}
\chi_{2,\FF}(\Delta) \leq (k_v^{\leq m}+71)(m+1)\Delta^{\frac{m}{m-1}}   
\end{equation}
\begin{equation}\label{gonc2}
\chi_{2,\FF}(\Delta) \leq (k_e^m+1+o(1))(m+1)\Delta^{\frac{m}{m-1}}   
\end{equation}

In this paper, we study a variation of the star-coloring problem and a special case of the 
$(2,\FF)$-subgraph coloring problem described above.  
We consider proper colorings of graphs not containing a bicolored (2-colored) 
path of a fixed length. 
We call a proper vertex
coloring of a graph $G$ without a bicolored copy of $P_k$ 
a {\it $P_k$-coloring} of $G$, where $k\geq 4$.  
The minimum number of colors needed for a $P_k$-coloring 
of $G$ is called {\it $P_k$-chromatic number} of $G,$ 
denoted by $s_k(G).$ 
A special case of this coloring is the 
star coloring, when $k=4,$ introduced by Gr\"unbaum \cite{grunbaum1973acyclic}. 
Hence, $\chi_S(G)=s_4(G)$ and all of the bounds on $s_k(G)$ 
in Section~\ref{sec:generalities} apply to star chromatic number using $k=4.$

%As a weakening of this problem and and acyclic coloring, 
We also study proper colorings that avoid bicolored long cycles. 
In particular, we let $\CCk=\{C_i: i\geq k\}$ 
and call a proper coloring of $G$ without any bicolored 
member from $\CCk$ a {\it $C_k$-coloring} of $G,$ 
and minimum number of colors needed for such a coloring 
{\it $C_k$-chromatic number} of $G,$ denoted by $a_k(G).$

Our results comprise lower bounds on these colorings and an upper bound 
for general graphs. 
In Section~\ref{sec:generalities}, we provide general 
bounds on $\sk$ and $\ak$ for any graph $G.$ 
In particular, our result presented in Theorem \ref{upperboundofP_k} improves on the bounds in \eqref{gonc1} and \eqref{gonc2}, for the case when $\FF=\{P_k\}$ for a fixed $k$. 
In Section~\ref{sec:product-coloring}, we present 
exact results on the $P_5$-coloring and $P_6$-coloring 
for the products of some paths and cycles. 
Finally, in Section \ref{sec:conc}, we conclude with some open problems.

\section{General Bounds} \label{sec:generalities}

%The {\it Tur\'an number of a graph F}, denoted by $\ex(n,F)$ is the maximum number of edges in a graph on $n$ vertices that does not contain $F$ as a subgraph. 
We obtain lower 
bounds on $s_k(G)$ and $a_k(G)$ by using 
a theorem of Erd{\H o}s and Gallai below. 
\begin{thm} \cite{erdHos1959maximal} \label{ErdosPkCk}
For a graph $G$ on $n$ vertices, if the number of edges is more than 
\begin{enumerate}
    \item 
    $\frac{1}{2}(k-2)n$,  
 then $G$ contains $P_k$ as a subgraph,
 \item
  $\frac{1}{2}(k-1)(n-1)$, 
then $G$ contains a member of $\CCk$ as a subgraph,
\end{enumerate}
for any $P_k$ with $k\geq 2$, and for any $\CCk$ with $k\geq 3.$
\end{thm}
Note that in a $P_k$-coloring of a graph, 
the subgraphs induced by any two color classes are $P_k$-free. 
Using this observation together with Theorem~\ref{ErdosPkCk}, 
we obtain the results in Theorems~\ref{lowerboundofP_k} and 
\ref{theoremC_k}. 
\begin{thm} \label{lowerboundofP_k}
For any graph $G=(V,E)$, let $|V|=n$ and $|E|=m$. Then, $s_k(G) \geq \frac{2m}{n(k-2)}+1$, for any $k\geq 3$.
\end{thm}
\begin{proof}
Let $s_k(G)=x$ and for $1 \leq i \leq x$, $V_i$ be the set of vertices whose color is $i$ in a $P_k$-coloring of G using $x$ colors. By the definition of $P_k$-coloring, the subgraph of $G$ induced by $V_i \cup V_j$ does not contain $P_k$ for any $i,j$. Let $E_{i,j}$ be the set of edges covered by vertices in the set $V_i \cup V_j$ with $1 \leq i < j \leq x$. Then,
\begin{align*}
  |E_{i,j}| \leq \frac{1}{2}(k-2)|V_i \cup V_j|,
%\hspace{2mm} k \geq 2   
\end{align*}
by Theorem \ref{ErdosPkCk}. For two distinct pairs $(i_1,j_1)$ and $(i_2,j_2)$ with $1 \leq i_1 < j_1 \leq x$ and $1 \leq i_2 < j_2 \leq x$, $E_{i_1,j_1} \cap E_{i_2,j_2}= \emptyset$. It also holds that
\begin{align*}
    \sum_{(i,j)}|E_{i,j}|=m \leq \sum_{(i,j)} \frac{1}{2}(k-2)(|V_i|+|V_j|)\leq 
 \frac{1}{2}(k-2)(x-1) \sum_{\ell=1}^{x} V_\ell  = \frac{n}{2}(k-2)(x-1).    
\end{align*}
The last inequality follows from the fact that 
each color $1\leq \ell \leq x$ appears $(x-1)$ times over pairwise distinct pairs of colors and $\sum_{\ell=1}^{x} V_\ell=n$. 
\end{proof}

As there are graphs such as the $d$-dimensional hypercube $Q_d,$ 
for which $\lceil \frac{d+3}{2}\rceil \leq s_4(Q_d)\leq d+1$ 
(\cite{fertin2004star}), 
the lower bound above is asymptotically correct in terms of the maximum degree $d$. 
However, we suspect that the coefficient $\frac{1}{k-2}$ can be improved.

\begin{thm}\label{theoremC_k}
For any graph $G=(V,E)$, let $|V|=n$, $|E|=m$ and 
%$\Delta = 8(\binom{n}{2}-m)+8m(\frac{2}{k-1}-1)+1.$
$\Delta=4n(n-1)-\frac{16m}{k-1}+1$, where $\Delta=\Theta(n^2)$ for $k\geq 4,$ and  $\Delta=\Theta(\binom{n}{2}-m)$ for $k=3$. 
Then, $a_k(G) \geq \frac{1}{2}(2n+1-\sqrt{\Delta})$, for any $k\geq 3$. 
\end{thm}
\begin{proof}
Let $a_k(G)=x$ and for $1 \leq i \leq x$, $V_i$ be the set of vertices whose color is $i$ in a $k$-acyclic coloring of $G$ using $x$ colors. 
%By the definition of $k$-acyclic coloring, the subgraph of $G$ induced by $V_i \cup V_j$ does not contain any cycle of length at least $k$, for any $1 \leq i < j \leq x$. Let $e_{i,j}$ be the set of edges covered by vertices in the set $V_i \cup V_j$. Then,
%\begin{align*}
%  |e_{i,j}| \leq \frac{1}{2}(k-1)(|V_i \cup V_j|-1) \text{, } \hspace{2mm} k \geq 3   
%\end{align*}
%For two distinct pairs $(i_1,j_1)$ and $(i_2,j_2)$ with $1 \leq i_1 < j_1 \leq x$ and $1 \leq i_2 < j_2 \leq x$, $e_{i_1,j_1} \cap e_{(i_2,j_2) }= \emptyset$. It also holds that
By the same ideas as in the proof of Theorem~\ref{lowerboundofP_k} and 
by Theorem \ref{ErdosPkCk}, we have
\begin{align*}
%    \sum_{(i,j)}|e_{i,j}|=
m \leq \sum_{(i,j)}
\frac{1}{2}(k-1)(|V_i|+|V_j|-1) = 
\frac{k-1}{4}[2n(x-1)-x(x-1)], 
%\frac{1}{2}(k-1) \left( (x-1) \left(\sum_{\ell=1}^{x} V_\ell \right) - \frac{x(x-1)}{2} \right).
\end{align*}
%There are $\frac{1}{2}x(x-1)$ pairwise distinct pairs of colors and each color $1\leq \ell \leq x$ appears $(x-1)$ times over those pairs of colors. Moreover, $\sum_{\ell=1}^{x} V_\ell=n$. So, we find that
which gives 
%\begin{align*}
$0 \geq x^2-(2n+1)x +(2n+\frac{4m}{k-1}).$
%\end{align*}
Let $\Delta = 4n^2-4n-\frac{16m}{k-1}+1.$ 
We have $\Delta\geq 1,$ since 
$k \geq 3$ and  $m \leq \frac{n(n-1)}{2}$. 
Thus, we have $x\geq \frac{1}{2}(2n+1-\sqrt{\Delta}).$

Note that $\Delta$ can also be written as 
$\Delta= 8(\binom{n}{2}-m)+8m(1-\frac{2}{k-1})+1.$ Since $\max\{\binom{n}{2}-m,m\} \geq \frac{1}{2}\binom{n}{2},$ asymptotically we have $\sqrt{\Delta}=\Theta(n)$ for $k\geq 4.$ When $k=3$, the second term disappears and $\Delta$ is of the order of $(\binom{n}{2}-m),$ the number of nonedges in $G.$
\end{proof}

%\subsection{Upper Bound} \label{sec:asymtotic}

In the following, we prove a general upper bound 
for the $P_k$-chromatic number of graphs. 
%In particular, we show that for any graph $G$ with maximum degree $d\geq 2,$ and 
%$k\geq4,$ $\sk\leq \lceil 6\sqrt{10}d^{\frac{k-1}{k-2}} \rceil.$ 
This coloring is also studied in \cite{esperet2013acyclic} 
as {\it star k} coloring, where a bicolored $P_{2k}$ is avoided 
in a proper vertex coloring, considering only paths of even order. 
Esperet et al.~\cite{esperet2013acyclic} show that 
every graph with maximum degree $\Delta$ has a star $k$ coloring 
($P_{2k}$-coloring) with at most $c_kk^\frac{1}{k-1}\Delta^\frac{2k-1}{2k-2}+
\Delta$ colors, where $c_k$ is a function of $k.$ 
Our result in Theorem \ref{upperboundofP_k} improves this result slightly and generalizes it to the paths of all lengths.    
%Independently, it is communicated \cite{rosenfeld2020comm} 
%that one can obtain the same bound using  
%the methods presented in \cite{rosenfeld2020another}, 
%and possibly the entropy compression method.
As discussed earlier, there are various methods such as 
Rosenfeld Counting method that can be used in 
proving upper bounds on such chromatic numbers \cite{rosenfeld2020another}. 
Our proof relies on Lovasz Local Lemma. 

%\begin{thm}[Mutual Independence] \label{Mutual}
An event $A_i$ is \textit{mutually independent} of a set of events $\{B_i \mid i=1,2...,n\}$ if for any subset $\mathcal{B}$ of events or their complements contained in $\{B_i\}$, we have $Pr[A_i \mid \mathcal{B}]=Pr[A_i]$.
%\subsection{Dependency Graph} \label{Dependency}
Let $\{A_1,A_2,...,A_n\}$ be events in an arbitrary probability space. A graph $G=(V,E)$ on the set of vertices $V=\{1,2,...,n\}$ is called a \textit{dependency graph} for the events $A_1,A_2,...,A_n$ if for each i, $1 \leq i \leq n$, the event $A_i$ is mutually independent of all the events $\{A_j \mid (i,j) \notin E\}$. 
\begin{thm}[General Lovasz Local Lemma] \label{GeneralLovaszLocal} \cite{erdHos1973problems}
 Suppose that $H=(V,E)$ is a dependency graph for the events $A_1,A_2,...,A_n$ and suppose there are real numbers $y_1,y_2,...,y_n$ such that $0 \leq y_i < 1$ and 
\begin{equation}\label{eq:LLL}
 Pr[A_i] \leq y_i \prod_{(i,j) \in E} (1-y_j)
 \qquad \text{for all $1 \leq i \leq n$.}
 \end{equation}
 Then $Pr[\bigwedge_{i=1}^n \Bar{A_i}] \geq \prod_{i=1}^n (1-y_i)$. In particular, with positive probability no event $A_i$ holds.
\end{thm}

We use Theorem~\ref{GeneralLovaszLocal} in the proof of the following upper bound.  
 \begin{thm}\label{upperboundofP_k}
Let G be any graph with maximum degree d. Then $s_k(G) \leq \lceil 6\sqrt{10}d^{\frac{k-1}{k-2}} \rceil$, for any $k\geq4$ and $d\geq2$. 
%In particular, $s_4(G)=\chi_S(G) \leq \lceil 6\sqrt{10}d^{3/2} \rceil$.
\end{thm}
\begin{proof}
Assume that $x=\lceil ad^{\frac{k-1}{k-2}} \rceil$ and $a=6\sqrt{10}$. Let $f:V \mapsto  \{1,2,...,x\}$ be a random vertex coloring of $G$, where for each vertex 
$v \in V$, the color $f(v) \in \{1,2,...,x\}$ is  chosen uniform randomly. It suffices to show that with positive probability $f$ does not produce a bicolored $P_k.$ 

%So, we need to define undesired events to apply Theorem \ref{GeneralLovaszLocal} and to show that probability of complement of these events is bigger than zero. After that the theorem will be proved. Let us describe two types of events we have chosen.
Below are the types of probabilistic events that are not allowed:
\begin{itemize}
    \item {Type I:} For each pair of adjacent vertices $u$ and $v$ of $G$, let $A_{u,v}$ be the event that $f(u)=f(v)$. 
    \item {Type II:} For each $P_k$ called $P$, let $A_P$ be the event that $P$ is colored properly with two colors. 
\end{itemize}
By definition of our coloring, none of these events are allowed to occur. 
%If none of these two undesired events occur, $f$ is a $P_k$-coloring of $G$. 
We construct a dependency graph $H$, where the vertices are the events of Types I and II, 
and use Theorem \ref{GeneralLovaszLocal} to show that with positive probability 
none of these events occur. 
For two vertices $A_1$ and $A_2$ to be adjacent in $H,$ the subgraphs corresponding to these 
events should have common vertices in $G.$
The dependency graph of the events is called $H,$ where  
the vertices are the union of the events. 
We call a vertex of $H$ {\it of Type i} if it corresponds to an event of Type i. 
%$A_?$ and $A_{P_k}$ in which two vertices representing two of our events are adjacent if and only if the sets of vertices of $G$ corresponding to two events share at least one common vertex. $H$ is dependency graph for the events considered since the occurrence of each $A_{P_k}$ or $A_{\{u,v\}}$ depends only on the colors of the vertices of $P_k$ or the colors of $u$,$v$ respectively. Now, we want to estimate the degree of vertices in $H$.
%\begin{obs}
For any vertex $v$ in $G$, there are at most 
\begin{itemize}
\item $d$ pairs $\{u,v\}$ associated with an event of Type I, and 
\item $\frac{k+1}{2}d^{k-1}$ copies of $P_k$ containing $v$, associated with an event of Type II.
\end{itemize}
%\end{obs}
%\begin{proof}
The first observation follows from the fact that $\Delta(G)=d.$ 
% That means existence of at most d pairs of adjacent vertices including $v$. 
To see the second observation, let us label the vertices of a $P_k$ 
containing $v$ as $x_1,x_2,...,x_k.$ 
%with the edges $\{x_i,x_{i+1}\},$ $1\leq i\leq k-1.$ 
The maximum number of $P_k$'s with $x_i=v$ is $d^{k-1}.$ 
Hence, there are at most $\lceil \frac{k}{2}\rceil d^{k-1}$ copies 
of $P_k$ containing $v$.  
%\end{proof}
\begin{lemma}
The $(i,j)^{th}$ entry of the following matrix shows an upper bound 
on the number of vertices of type $j$ that are adjacent to 
a vertex of type $i$ in $H.$
\begin{center}
\[ 
\begin{array}{|c|c|c|} 
\hline
\empty & I & II \\
\hline
I & 2d  & (k+1)d^{k-1}\\ 
\hline
II & kd & {\frac{k}{2}}(k+1)d^{k-1}\\
\hline
\end{array} 
\] 
\end{center}
\end{lemma}
%\begin{proof} 
Consider a vertex $A_{u,v}$ in $H$ for the 
first row. 
Since this vertex may be adjacent to events 
$A_{u,z}$ and $A_{v,x}$ for some $x,z\in V(G),$  there are at most $2d$ such events of Type I. 
Similarly, $A_{u,v}$ may be adjacent 
to events $A_P,$ where $P$ is a $P_k$ containing $u$ or $v.$ There are at most $(k+1)d^{k-1}$ 
such events. 
For the second row, a path $P$ that is a copy of $P_k$ may have $kd$ vertices that are 
adjacent to some vertex of $P.$  
Similarly, there may be at most $(k+1)d^{k-1}/2$ other $P_k$'s containing 
some particular vertex of $P.$ 
%\end{proof}
The probabilities of the events are 
 \begin{itemize}
     \item $Pr(A_{u,v})=\frac{1}{x}$ for an  event of type I, and   
     \item $Pr(A_P)=\frac{1}{x^{k-2}}$ for an  event of type II.     
 \end{itemize}

To apply Theorem \ref{GeneralLovaszLocal}, 
we choose the weights $y_i$ as below:
 \begin{align*}
     y_1=\frac{1}{3d}, \qquad y_2=\frac{1}{2(k+1)d^{k-1}}.
 \end{align*}
%we can observe that $0 < y_{1,2} <1$. 
%By Theorem \ref{GeneralLovaszLocal}, it suffices to verify the following inequalities.
Below are the conditions that are to be satisfied for~\eqref{eq:LLL} to hold. 
  \begin{equation} \label{ineq1}
        \frac{1}{x} \leq \frac{1}{3d}\left (1-\frac{1}{3d} \right)^{2d} \left(1-\frac{1}{2(k+1)d^{k-1}}\right)^{(k+1)d^{k-1}} 
  \end{equation}
    
    \begin{equation} \label{ineq2}
      \frac{1}{x^{k-2}} \leq \frac{1}{2(k+1)d^{k-1}} \left(1-\frac{1}{3d}\right)^{kd} \left(1-\frac{1}{2(k+1)d^{k-1}}\right)^
      {\frac{k}{2}(k+1)d^{k-1}} 
    \end{equation}
%Since $\lceil ad^\frac{k-1}{{k-2}} \rceil \geq ad^\frac{k-1}{{k-2}}$, it is enough to show that
%\begin{align*}
%    \frac{1}{ad^{{k-1}/{k-2}}} \leq \frac{1}{3d}\left(1-\frac{1}{3d}\right)^{2d}\left(1-\frac{1}{2(k+1)d^{k-1}}\right)^{(k+1)d^{k-1}} 
%    \end{align*}
%First, we show that~\eqref{ineq1} holds.    
Since $(1+z)^n\geq 1+nz$ for $z\geq -1$ and 
any nonnegative integer $n$, 
it is sufficient to verify the following to satisfy~\eqref{ineq1}, and   
we observe that it holds when $a=6\sqrt{10}\geq 18$ and $k\geq 4.$ 
\begin{align*}
    \frac{1}{ad^\frac{k-1}{k-2}} \leq \frac{1}{3d}\left(1-\frac{2d}{3d}\right)\left(1-\frac{(k+1)d^{k-1}}{2(k+1)d^{k-1}}\right) =  \frac{1}{18d}
\end{align*}

We can rewrite~\eqref{ineq2} as below.  
\begin{align*}
      \frac{1}{a} \leq 
      \left (\frac{1}{2(k+1)}\right)^\frac{1}{k-2}
      \left(1-\frac{1}{3d}\right)^\frac{kd}{k-2}
\left(1-\frac{1}{2(k+1)d^{k-1}}\right)^
{\frac{k(k+1)d^{k-1}}{2(k-2)}} 
\end{align*}
%Since $(1+x)^n\geq 1+nx$ for $x\geq -1$ and any nonnegative integer $n$, 
Similarly, it is sufficient to verify the following to satisfy~\eqref{ineq2}. We omit the use of ceiling for simplicity.  
\begin{align*}
      \frac{1}{a} \leq 
      \left (\frac{1}{2(k+1)} \right)^{1/{k-2}}
      \left(1-\frac{k}{3(k-2)}\right) \left(1-\frac{k}{4(k-2)}\right)
\end{align*}
Since all factors on the right are decreasing for $k\geq 4$, ~\eqref{ineq2} is verified. 
%$\frac{1}{6\sqrt{10}} \leq \frac{1}{\sqrt{10}} (1-\frac{4}{6}) (1-\frac{4}{8})$
%Both of $\frac{k}{3(k-2)}$ and $\frac{k}{4(k-2)}$ have global minimum in the interval $[4,\infty)$ and the minimum values for those functions occur at $k=4$. In the same way, ${2^{1/{k-2}}}$ and $(k+1)^{1/{k-2}}$ are decreasing functions in the interval $[4,\infty)$ and  the maximum values are $\sqrt{2}$ and $\sqrt{5}$, respectively. Then, we get
\end{proof}

\section{Coloring of Products of Paths and Cycles} \label{sec:product-coloring}
The \textit{cartesian product} of two graphs $G=(V,E)$ and $G'=(V',E')$ is shown 
by $G \square G'$ and its vertex set is $V \times V'$. 
For any vertices $x,y \in V$ and $x',y' \in V'$, there is an edge between 
$(x,y)$ and $(x',y')$ in $G \square G'$ if and only if 
either $x=y$ and $x'y' \in E'$ or $x'=y'$ and $xy \in E$. 
For simplicity, we let $G(n,m)$ denote the product $P_n\square P_m.$ 

The star chromatic number and acyclic chromatic number of products 
of graphs have been widely studied. 
Fertin et al.~\cite{fertin2004star} find various bounds on the star chromatic number 
of graph families such as hypercube, grid, tori, providing exact results for 2-dimensional grids. 
More recent results on the acyclic coloring of grid and tori are presented by  
Akbari et al.\cite{akbari2019star} and Han et al.~\cite{han2016star}. 
Jamison et al.~\cite{jamison2008acycliccycles,jamison2008acyclic,jamison2006acyclic} investigate the acyclic chromatic number for products of trees, 
products of cycles and Hamming graphs. 
However, these numbers are still not known in general for the  products of several cycles. 

In this section, we present results on $P_5$-coloring and $P_6$-coloring of products of paths and cycles of various lengths. We first present Theorem \ref{skG33C33}, which helps us to prove more general cases by providing a lower bound.  

\begin{thm} \label{skG33C33}
\[
s_5(P_3 \square P_3)=
s_5(C_3\square C_3) = 
s_5(C_3\square C_4) = 
s_5(C_4\square C_4) = 4.
\]
\end{thm}
\begin{proof} 
We start by showing that $s_5(P_3 \square P_3)\geq 4.$
Assume that there is a coloring of 
$P_3 \square P_3$ using colors $\{a,b,c\}$ only. We want to show that each color appears exactly twice in any consecutive columns. Note that each color appears at most 3 times in consecutive columns. If a color, say $a$, appears 3 times, then by the pigeonhole principle,  a color, say $c$, appears exactly once on these consecutive columns. In this case, the vertices colored $a$ and $b$ contain a bicolored $P_5$.  Hence, each color is used exactly twice and all colors appear in any consecutive columns. 

Without loss of generality, suppose that $a$ 
is used twice in a column.  
Then, in a consecutive column, either $b$ or $c$ is used twice, which is impossible in a proper 
coloring using $\{a,b,c\}$ only. 
Thus, each column has colors $a,b,c$ exactly once. This implies that any coloring of 
$P_3\square P_3$ using three colors has 
to be colored as shown below.  
According to this property, if 
the vertex at the center of 
$P_3\square P_3$ has, say color $a,$  
then some pair of vertices at opposing corners 
have color $a$ as well. 
As we see in~\eqref{C3xC3}, however 
the remaining vertices are colored, 
there is always a bicolored $P_5$, thus  
$s_5(P_3 \square P_3)\geq 4.$
\begin{equation}\label{C3xC3}
\begin{matrix}   
    a & b & c\\
    c & a & b\\
    b & c & a
\end{matrix}    
%\qquad \qquad
%\begin{matrix}
%c & b & a\\
%b & a & c\\
%a & c & b\\
%\end{matrix}
\qquad \qquad \qquad
\begin{matrix}
1 & 2 & 3 & 4\\
2 & 1 & 4 & 3\\
3 & 4 & 1 & 2\\
4 & 3 & 2 & 1
\end{matrix} 
\end{equation}
%\begin{figure}[htbp]
%    \centering
%    \includegraphics[scale=1]{figLLL/fig2_LLL.tex}
%    \caption{Possible two positions of colors on $G(3,3)$}
%    \label{fig2_LLL}
%\end{figure}
Since $C_3 \square C_3$, $C_3 \square C_4$ and 
$C_4 \square C_4$ contain $P_3 \square P_3$ as a subgraph, this shows that at least 4 colors 
are needed to color these graphs.  
Such a coloring can be obtained as  
in~\eqref{C3xC3} by 
taking the first three or four rows/columns 
depending on the change in the grid dimension. 
\end{proof}

\begin{thm} \label{s_5(G(n,m))}
$s_5(G(n,m))=4$ for all $n,m \geq 3$.
\end{thm}
\begin{proof}
Note that $4=s_5(G(3,3)) \leq s_5(G(n,m))$ for all $m,n \geq 3.$ Since there exists some integer $k$ for which $3k \geq n,m$ and $G(n,m)$ is a subgraph of $G(3k,3k)$, $s_5(G(n,m)) \leq s_5(G(3k,3k))$ for some $k$. Hence, we show 
that $s_5(G(3k,3k))=4$. 
In Theorem~\ref{skG33C33}, a $P_5$-coloring of $C_3\square C_3$ is 
given by the upper left corner of the coloring in~\eqref{C3xC3} 
by using 4 colors. 
By repeating this coloring of $C_3\square C_3$ $k$ times in $3k$ rows, we obtain 
a coloring of $G(3k,3).$ 
Then repeating this colored $G(3k,3)$  $k$ times in $3k$ columns, 
we obtain a $P_5$-coloring of $G(3k,3k)$ using 4 colors. 
%This operation is actually adding copies of~\eqref{C3xC3} one under the other and side by side. 
There exists no bicolored $P_5$ in this coloring.  
\end{proof}

In the following, we finally generalize our results to the product of all cycles with some exception. We make use 
of the well-known result below. 
\begin{thm} [Sylvester, \cite{sylvester1884mathematical}]\label{sylvester}
If $r, s > 1$ are relatively prime integers, 
then there exist $\alpha , \beta \in \mathbb{N}$ such that 
$t=\alpha r + \beta s$ for all $t \geq (r - 1)(s -1)$. 
\end{thm}

\begin{thm} \label{cixcj}
Let $p,q \geq 3$ and $p,q \neq 5$. Then $s_5(C_p \square C_q)=4$.
\end{thm}
\begin{proof}
The lower bound follows from Theorem \ref{skG33C33}. 
%For upper bound, we produce patterns of $C_3 \square C_j$, $C_4 \square C_j$ and $C_i \square C_j$ ($i,j \geq 6$) copying these patterns. \\
By Theorem \ref{sylvester},  $p$ and $q$ can be written as a linear combination of 3 and 4 
using nonnegative coefficients.  
%for all $i,j\geq(3-1)(4-1)=6$. 
%Hence, let $p =3a_1+4b_1$ and $q=3a_2+4b_2$ for some nonnegative $a_i, b_i.$ 
By using this, we are able to tile the $p\times q$-grid of 
$C_p\square C_q$ using these 
blocks of $3\times 3,$ $3\times 4,$ $4\times 3,$ and 
$4\times 4$ grids. 
Recall that the coloring pattern in \eqref{C3xC3} also provides a  
$P_5$-coloring of smaller grids listed above by using 
the upper left portion for the required size. 
Therefore, using these coloring patterns on the  
smaller blocks of the tiling yields a $P_5$-coloring of $C_p\square C_q.$
%such as $C_3 \square C_4$, $C_4 \square C_3$ and $C_3 \square C_3$, 
%call these blocks $A$, $B$, $C$, respectively, and call the 
%entire coloring pattern $D.$ 
\end{proof}

\begin{corollary} \label{pixcj}
Let $i,j \geq 3$ and $i,j \neq 5$. Then, $s_5(P_i \square C_j)=4$. 
\end{corollary}
\begin{proof}
Since $P_i \square P_j$ is a subgraph of $P_i \square C_j,$ 
Theorem~\ref{s_5(G(n,m))} together with Theorem~\ref{cixcj} yields this result. 
\end{proof}

Finally, we make use of the $P_5$-chromatic numbers found earlier to 
prove similar results for $P_6$-coloring as well, since every $P_5$-coloring is also a $P_6$-coloring. 
\begin{thm} \label{s6G44}
$s_6(G(4,4))=4.$ 
\end{thm}
\begin{proof}
Since $s_6(G(4,4)) \leq s_5(G(4,4))=4$, 
we prove $s_6(G(4,4))\geq 4$. Assume that $f$ 
is a coloring of $G(4,4)$ using 
the colors $\{1,2,3\}$ only. We consider 
possible colorings on the $C_4$ at the center 
of the grid, call it $C.$ 
\paragraph{Case 1: $C$ is bicolored.} 
\begin{figure}[htbp]
    \centering
    \includegraphics[scale=0.8]{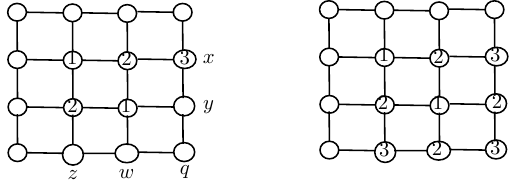}
    \caption{Possible colorings in Case 1.}
    \label{fig11_LLL}
\end{figure} 
Assume that $C$ has only two colors, 1 and 2. 
Then, either $x$ or $y$ shown in Figure~\ref{fig11_LLL} has color 3. 
Without loss of generality, assume that $f(x)=3.$ This implies $f(y)=2$. 
To avoid a bicolored $P_6,$ we have $f(q)=3.$ 
This implies that $f(w)=2$, and therefore $f(z)=3$ so that $V(C) \cup \{z,w\}$ is not bicolored. However, this yields a bicolored $P_7$ as seen in Figure \ref{fig11_LLL}.
 
\paragraph{Case 2: $C$ has all three colors.} 
We assume that the repeating color on $C$ is 1.

\paragraph{ Case 2.a: Color 1 is also used on the pair of vertices in opposing corners as in Figure~\ref{fig12_LLL}.}

Note that $x$ and $y$ cannot have the same color, otherwise there is a bicolored $P_6.$  Same holds for $w$ and $z.$ 
Hence, both 2 and 3 appear as colors on 
the pairs $\{x,y\}$ and $\{w,z\},$ 
yielding a bicolored $P_6.$
\begin{figure}[h!]
  \centering
  \begin{subfigure}[b]{0.33\linewidth}
  \centering
       \includegraphics[scale=0.8]{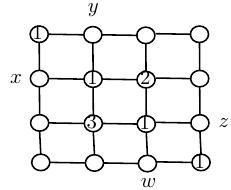}
    \caption{}
    \label{fig12_LLL}
  \end{subfigure}
  \begin{subfigure}[b]{0.62\linewidth}
  \centering
  \includegraphics[scale=0.8]{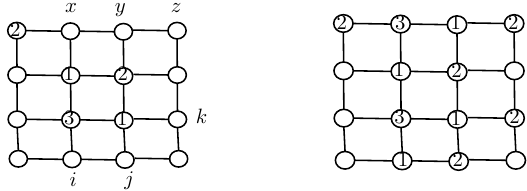}
 \caption{}
 \label{fig13_LLL}
  \end{subfigure}
  \caption{Possible colorings in Case 2.}
  \label{fig_gen}
\end{figure}
 %\begin{figure}[htbp]
 %   \centering
 %   \includegraphics[scale=0.8]{figLLL/fig12_LLL.tex}
%    \caption{}
%    \label{fig12_LLL}
%\end{figure}
\paragraph{Case 2.b: Color 1 is not used on both of the vertices in opposing corners as in Figure~\ref{fig12_LLL}.} Without loss of generality, assume that one of the vertices at the corners is colored 2 as in Figure \ref{fig13_LLL}. 
This case is also symmetric to the case when this color is 3. 
This implies that $f(x)=3$ and $f(y)=1$ yielding  a bicolored $P_5$. To avoid a bicolored (with colors 1 and 3) $P_6$, it is necessary that $f(j)=f(k)=f(z)=2$. However, this produces a  bicolored $P_6$ seen in Figure \ref{fig13_LLL}. 
\end{proof}

\begin{corollary}\label{s_6(G(m,n))}
$s_6(G(n,m))=4$ for all $n,m \geq 4$.
\end{corollary}
\begin{proof}
Theorem \ref{s_5(G(n,m))} together with  Theorem~\ref{s6G44} imply this result. 
\end{proof}

\begin{corollary}\label{cor_s6_cmcn}
$s_6(C_m \square C_n)=4$ for all $m,n \geq 4$ and $m,n \neq 5$.
\end{corollary}
\begin{proof}
By the definition of $P_k$-coloring and Theorem \ref{cixcj}, 
$s_6(C_m \square C_n) \leq s_5(C_m \square C_n)=4$ for all $m,n \geq 3$ and $m,n \neq 5$. 
Since $G(4,4)$ is a subgraph of $C_m \square C_n$ for all $m,n \geq 4$ and by Corollary~\ref{s_6(G(m,n))}, $s_6(C_m \square C_n) \geq s_6(G(4,4))=4$. 
%Hence, $s_6(C_m \square C_n)=4$ for all $m,n \geq 4$ and $m,n \neq 5$.
\end{proof}

\section{Conclusions}\label{sec:conc}

In this paper, we introduce a variation of the star coloring 
problem by forbidding a family different than double stars 
to be bicolored, such as paths of a fixed length and cycles 
that are longer than some fixed integer $k.$ 
We discuss below some questions that remain open. 

We observe that $s_k(G(2,m))=3$ for all $m \geq 3$, $k\geq 5$, and 
$s_k(G(3,m))=3$ for all $m \geq 3$ and $k \geq 6$ 
by using the 3-colorings shown below. 
%in~\eqref{eq:sk_G3m}. 
\begin{equation*}\label{eq:sk_G3m}
\begin{matrix}
1& 2 & 3& 1& 2 & 3&...\\
2 & 3 & 1& 2 & 3 & 1&...\\
\end{matrix}\qquad \qquad \qquad
\begin{matrix}
1 & 2 & 3 & 1 & 2 & 3 & \dots\\
2 & 3 & 1 & 2 & 3 & 1 & \dots\\
1 & 2 & 3 & 1 & 2 & 3 & \dots\\
\end{matrix}
\end{equation*} 
This coloring pattern having  
alternating rows repeating the patterns 123 and 231  
%columns with color pairs $(1,2), (2,3), (3,1),$ respectively, 
can be generalized to show that $s_k(G(k-3,m))=3$ for any $k\geq 6$ and $m\geq 3$.  
Note that in such colorings, a bicolored $P_k$ has to have at least $k-2$ vertices in the same column, thus cannot be found in $G(k-3,m)$ colored according to the pattern above. As we know by Theorem \ref{s_5(G(n,m))} and Corollary \ref{s_6(G(m,n))} that 
$s_5(G(n,m))=4$ for all $n,m \geq 3$, and 
$s_6(G(n,m))=4$ for all $n,m \geq 4$, it seems likely that this result can be generalized as in the following claim, which would settle all cases for $s_k(G(n,m)).$ 
\begin{conj}\label{conj_sk}
$s_k(G(k-2,k-2))=4$. Thus, $s_k(G(n,m))=4$ 
for all $n, m\geq k-2.$
\end{conj} 

It also remains as an open question, whether one can find coloring patterns as above, to show similarly that $s_k(C_{k-3}\square C_q))=3$ for any $k\geq 5$, $q\geq 3$, and prove a similar statement for $s_k(C_p\square C_q)$ as above.   
%We believe that similar to Conjecture \ref{conj_sk}, one can generalize Theorem \ref{cixcj} and Corollary \ref{cor_s6_cmcn} for all $k$ as $s_k(C_p\square C_q)=4$ for all $p\geq k-2,$ $q\geq 3$ ($p,q\neq 5$).

Finally, the upper bound in Theorem~\ref{upperboundofP_k} also holds for $a_k(G)$ of any graph $G,$ since $a_k(G)\leq s_k(G)$. It would be interesting to study whether this bound can be reduced further for $a_k(G).$

\bibliographystyle{amsplain}
\bibliography{arxiv_submission2023}
\end{document}